\documentclass[11pt,a4paper]{article}
\usepackage{amsmath,amsthm,amsopn,amsfonts,graphicx,amssymb,dsfont,color}
\usepackage{mathrsfs}
\usepackage{color}

\usepackage{enumerate}

\usepackage{a4,a4wide}

\def\ep{{\varepsilon}}
\def\R{\mathbb R}
\def\dir{\omega}
\def\P{\mathcal P}

\newtheorem{theo}{\textbf{Theorem}}[section]

\newtheorem{example}[theo]{\it{Example}}

\newtheorem{prop}[theo]{\textbf{Proposition}}

\newtheorem{cor}[theo]{\textbf{Corollary}}

\newtheorem{rem}[theo]{\textbf{Remark}}

\title{Evolution equations involving nonlinear truncated Laplacian operators}
\date{}

\begin{document}
\maketitle

\begin{center}
{\large\bf Matthieu Alfaro \footnote{IMAG, Univ. Montpellier, CNRS, Montpellier, France. E-mail: matthieu.alfaro@umontpellier.fr} and Isabeau Birindelli \footnote{Dipartimento di Matematica, Sapienza Universit\`{a}, Rome, Italy. E-mail:
isabeau@mat.uniroma1.it}.} \\
[2ex]
\end{center}


\tableofcontents

\vspace{10pt}

\begin{abstract} We first study the so-called Heat equation with two families of elliptic operators which are fully nonlinear, and depend on some eigenvalues of the Hessian matrix. The equation with operators including the \lq\lq large'' eigenvalues has strong similarities with  a Heat equation in lower dimension whereas, surprisingly, for operators including \lq\lq small'' eigenvalues it shares some properties with some transport equations. In particular, for these operators, the Heat equation (which is nonlinear) not only does not have the property that \lq\lq disturbances propagate with infinite speed'' but may lead to quenching in finite time. Last, based on our analysis of the Heat equations (for which we provide a large variety of special solutions) for these operators, we inquire on the associated Fujita blow-up phenomena.\\


\end{abstract}

\section{Introduction}\label{s:intro}

Let $N\geq 2$ be given. For $u:\R^N\to \R$, say of the class $C^2$, we denote by
$$
\lambda_1(D^2u)\leq \cdots\leq \lambda _N(D^2u)
$$
the eigenvalues of the Hessian matrix $D^2u$. For $1\leq k<N$ we consider the fully nonlinear elliptic operators given by
\begin{equation}
\label{Pmoins}
\P _k^- u:=\sum_{i=1}^k \lambda_i(D^2u),
\end{equation}
and
\begin{equation}
\label{Pplus}
\P _k^+ u:=\sum_{i=N-k+1}^N \lambda_i(D^2u).
\end{equation}
Notice that the case $k=N$ leads to the linear situation $\P _N^{-}u=\P _N^{+}u=\Delta u$, hence we will always suppose that $k<N$.

Our first main goal is to understand  the Heat equations
\begin{equation}
\label{eq-heat-moins}
\partial _t u=\P _k ^{-}  u \quad \text{ in } (0,+\infty)\times \R^{N},
\end{equation}
and
\begin{equation}
\label{eq-heat-plus}
\partial _t u=\P   _k^{+} u\quad \text{ in } (0,+\infty)\times \R^{N},
\end{equation}
together with the associated Cauchy problems. As revealed below by our analysis, naming \eqref{eq-heat-moins} a Heat equation is controversial but, for the moment, we adopt this denomination.

Our second main goal is to analyze the  Fujita blow-up phenomena \cite{Fuj-66}, \cite{Wei-81}, \cite{Kav-87}, \cite{Alf-fuj}, for the Cauchy problems associated with equations
\begin{equation}
\label{eq-fujita-moins}
\partial _t u=\P _k^{-} u+u^{1+p} \quad \text{ in } (0,+\infty)\times \R^{N},
\end{equation}
and
\begin{equation}
\label{eq-fujita-plus}
\partial _t u=\P _k ^{+}u+u^{1+p} \quad \text{ in } (0,+\infty)\times \R^{N},
\end{equation}
where $p>0$.

\medskip

Let us mention that these highly degenerate elliptic operators have been introduced in the context of differential 
geometry, by Wu \cite{WuH} and Sha \cite{ShaJP}, in order to solve problems related to manifolds with partial 
positive curvature. In a related fashion, they appear in the analysis of mean curvature flow in arbitrary 
codimension performed by Ambrosio and Soner   \cite{Amb-Son-96}.

In the context of elliptic PDE they are already considered as an example of degenerate fully non linear operators 
in the {\it User's guide} \cite{Cra-Ish-Lio-92}, but more recently both 
Harvey and Lawson in \cite{HarvLawCPAM, HarvLawIUMJ} and Caffarelli, Li and Nirenberg \cite{CafLiNirIII} have 
studied them in a completely new light.

Finally, in the very last years, some new results have been obtained on Dirichlet problems in bounded domains in 
relationship with the convexity of the domain, through the study of the maximum principle and the so called \lq\lq principal eigenvalue",  see \cite{OberSilv} and  \cite{Bir-Gal-Ish-18, Bir-Gal-Ish-preprint}.

Notice also that, due to the links between the behavior of solutions to evolution equations (Fujita blow up phenomenon) and the existence of steady states (nonlinear Liouville theorems), see \cite{Gui-Ni-Wan-92} e.g., the works of Birindelli, Galise and Leoni \cite{Bir-Gal-Leo-17}  and  Galise \cite{Gal-19} can be seen as a starting point for the present paper focused on evolution problems.

We  also wish to mention the very recent works of Blanc and Rossi that study
degenerate elliptic operators defined by $\mathcal P _j u:=\lambda_j(D^2u)$ for some $1\leq j\leq N$.  In other words, instead of considering the sum of the 
$k$ smallest or $k$ largest eigenvalues, they consider only the $j$-th eigenvalue of the Hessian matrix.
Even though they are different operators they share some analogies both in the definitions and in the difficulties 
that arise in studying them. These authors have considered both the steady state equation \cite{BlaRos} and the evolution equation \cite{BlaEsteRos} in a bounded domain. They mainly focus on the well-posedness of such problems (in the viscosity sense), and their approximation by a two-player zero-sum game.

As far as we know, this work is the first analysis of {\it evolution equations} in $\R^{N}$  involving the aforementioned nonlinear 
truncated Laplacian operators. Since we explore many directions, and collect results that we believe to be of 
equal importance, we take the liberty not to present a section with some so-called main results. Instead, we give 
below a rather detailed overview of the paper.

In Section \ref{s:explicit} we compute naive explicit solutions to the Heat equations \eqref{eq-heat-moins} and 
\eqref{eq-heat-plus}. This can be seen as a warm-up, already revealing the importance of convexity/concavity of 
solutions.

In Section \ref{s:self-similar} we inquire on the existence of self similar solutions to \eqref{eq-heat-moins} and 
\eqref{eq-heat-plus}. A key observation is that when a solution $u$ is \lq\lq one dimensional'', its Hessian has an eigenvalue of multiplicity (at least) $N-1$ and therefore, computing $\P _k^{\pm} u$ reduces to localize the last eigenvalue, see assumption \eqref{hyp-self}. The outcomes are the following: for equation \eqref{eq-heat-moins} involving $\P _k ^-$, self similar solutions have algebraic decay as $\vert x\vert \to +\infty$; for equation \eqref{eq-heat-plus} involving $\P _k ^+$ self similar solutions are the Heat kernels in lower dimension $k<N$ which, in particular, have a $L^1(\R ^N)$ norm which is  increasing in time like $t^{\frac{N-k}{2}}$.

In Section \ref{s:well}, we quote a result of Crandall and Lions \cite{Cra-Lio-90} to obtain the global well-posedness, in the viscosity sense, of the Cauchy problems \eqref{eq-heat-moins} and \eqref{eq-heat-plus}, and the local well-posedness, of the Cauchy problems \eqref{eq-fujita-moins} and \eqref{eq-fujita-plus}, which we plan to study in the end of the paper.

In Section \ref{s:radial} we inquire on radial solutions to the Cauchy problems \eqref{eq-heat-moins} and \eqref{eq-heat-plus}. It turns out that the Heat equation \eqref{eq-heat-moins} involving $\P _k ^-$ may not diffuse but transports. In other words, the operator $\P _k^-$ shares some similarities with some first order operators. As a by product, we can construct a very surprising example of an initial data driving the solution to zero everywhere in finite time, see Example \ref{ex:quenching}, which is referred as a {\it quenching} phenomena. On the other hand, and as already suspected since Section \ref{s:self-similar}, the Heat equation \eqref{eq-heat-plus} in dimension $N$ involving $\P _k^+$ behaves like the Heat equation in lower dimension $k<N$.

Finally, Section \ref{s:fujita} is devoted to the analysis of the Cauchy problems \eqref{eq-fujita-moins} and \eqref{eq-fujita-plus}. We aim at determining the Fujita exponent $p_F$ separating  \lq\lq sytematic blow-up when $0<p<p_F$'' from \lq\lq existence of global solutions when $p>p_F$'' (see Section \ref{s:fujita} for a more precise statement). The proofs rely on the variety of special solutions to the Heat equations \eqref{eq-heat-moins} and \eqref{eq-heat-plus} collected in the previous sections. We prove that $p_F=0$ for \eqref{eq-fujita-moins}  involving $\P _k^-$, whereas $p_F=\frac 2 k$ for \eqref{eq-fujita-plus} involving $\P _k ^+$. These facts were highly suspected from the previous sections but the proofs for \eqref{eq-fujita-plus} are far from trivial: the proof of Theorem \ref{th:fujita-plus} in particular requires the combination of the comparison principle, the subtle solutions of Example \ref{ex:compactly} and a comparison between $\P _k^+ u$ and $\Delta u$ which is available for radial and smooth solutions.

Let us mention that, from places to places, we have indicated some directions and presented some
preliminary computations that lead to partial conclusions or observations, and therefore raise some open problems. We have also tried to underline the variety of possible behaviors of the evolution equations under consideration by providing many examples of very different solutions.

\section{Explicit solutions to the Heat equations}\label{s:explicit}

\subsection{Convex/concave functions of one variable}\label{ss:convex}

If, for some $1\leq i\leq N$,
$$
u(t,x)=\varphi (x_i), \quad \text{ with $\varphi$ a $C^{2}$ convex function,}
$$
then $u$ solves \eqref{eq-heat-moins}. Indeed $D^2u=\text{Diag } (0,0,...,\underbrace{\varphi''(x_i)}_{\text{ith position}},...,0,0)$. Since $\varphi''(x_i)\geq 0$ the $N-1$ smallest eigenvalues are 0, hence $\mathcal P_k^{-} u=\lambda_1(D^2u)+\cdots+\lambda_k(D^2u)=0=\partial _t u$.

\medskip

Similarly if, for some $1\leq i\leq N$,
$$
u(t,x)=\varphi (x_i), \quad \text{ with $\varphi$ a $C^{2}$ concave function,}
$$
then $u$ solves \eqref{eq-heat-plus}.

\subsection{One variable travelling waves}\label{ss:tw}
For some $1\leq i\leq N$, let
$$
u(t,x)=\varphi(x_i-ct),\quad \text{ with $c\neq0$.}
$$
If $\varphi$ is convex then, again, $\mathcal P _k^{-} u=0$ and thus we need $-c\varphi'=0$ and $\varphi=cst$ already found above. On the other hand, if $\varphi$ is concave then $\mathcal P_k^{-}u(t,x)=\varphi''(x_i-ct)$ and thus we need $-c\varphi '=\varphi ''$ that is
$\varphi(z)=\alpha-\beta e^{-cz}$ with $\beta>0$ to get the concavity. Hence we are equipped with 
$$
u(t,x)=\alpha-\beta e^{-c(x_i-ct)}, \quad \alpha\in \R, \beta>0, c\neq 0,
$$
solutions to \eqref{eq-heat-moins} that are planar travelling waves connecting $\alpha$ to $-\infty$.

\medskip

Similarly, we are equipped with 
$$
u(t,x)=\alpha+\beta e^{-c(x_i-ct)}, \quad \alpha\in \R, \beta>0, c\neq 0,
$$
solutions to \eqref{eq-heat-plus} that are planar travelling waves connecting $\alpha$ to $+\infty$.

\subsection{Polynomial solutions}\label{ss:polynomial}




For any $A\in \mathcal S_N(\R)$, any $x_0\in \R^N$, any $y\in \R^N$, any $C\in\R$,
$$
u(t,x)=\left(\sum_{i=1}^k \lambda_i(A)\right)t+\frac 12 A(x-x_0)\cdot(x-x_0)+(x-x_0)\cdot y +C
$$
solves \eqref{eq-heat-moins}, whereas 
$$
u(t,x)=\left(\sum_{i=N-k+1}^N \lambda_i(A)\right)t+\frac 12 A(x-x_0)\cdot(x-x_0)+(x-x_0)\cdot y +C
$$
solves \eqref{eq-heat-plus}, since in the two above cases $D^{2}u=A$. Those solutions provide the sub and supersolutions used in the proof of \cite[Theorem 2.7]{Cra-Lio-90}.

\section{The Heat equations: self similar solutions}\label{s:self-similar}

If $u(t,x)$ solves \eqref{eq-heat-moins} or \eqref{eq-heat-plus} so does $Cu(\lambda^{2}t,\lambda x)$, $C>0$, $\lambda>0$. We thus look after a nonnegative self similar solution in the form
\begin{equation}
\label{ansatz}
u(t,x):=\frac{1}{t^\beta}\varphi\left(\frac{\vert x\vert}{\sqrt t}\right), 
\end{equation}
for some $\varphi=\varphi(r)$, $\beta\in \R$, and where $\vert x\vert =(x_1^2+\cdots+x_N^2)^{1/2}$. We also require $\varphi(0)=1$, $\varphi'(0)=0$.

We immediately get 
$$
\partial _t u=-\frac{\beta}{t^{\beta+1}}\varphi -\frac 1 2 \frac{\vert x\vert}{t^{\beta+\frac 32}}\varphi '.
$$
Next, after straightforward computations, we obtain the Hessian matrix
$$
D^{2}u=\frac{1}{t^{\beta+\frac 12}}\left(\frac{1}{\vert x\vert}\varphi'Id_N-\left(\frac{1}{\vert x\vert}\varphi' -\frac{1}{t^{\frac 12}}\varphi'' \right)\frac{x}{\vert x\vert }\otimes\frac{x}{\vert x\vert}\right).
$$
Since $\frac{x}{\vert x\vert}\otimes \frac{x}{\vert x\vert}$ is a matrix of rank 1, $\frac{1}{t^{\beta+\frac 12}\vert x\vert}\varphi '$ is an eigenvalue of $D^{2}u$ with multiplicity (at least) $N-1$. By considering the traces of the matrices we see that the remaining eigenvalue has to be $\frac{1}{t^{\beta+1}}\varphi ''$. From now on, we assume
\begin{equation}
\label{hyp-self}
\varphi''(r)\geq \frac 1 r \varphi'(r), \quad \text{ for all } r>0,
\end{equation}
which enables to compute $\mathcal P^{\pm} _k u$. Notice that other assumptions than \eqref{hyp-self} will be discussed in subsection \ref{ss:on-assumption} and will reveal much less natural.

\subsection{Operator $\P _k^{-}$}\label{ss:self-similar-moins}

Under assumption \eqref{hyp-self}, we have
$$
\P ^{-}_ku=\sum _{i=1}^{k}\lambda_i(D^{2}u)=\frac{k}{t^{\beta+\frac 12}\vert x\vert}\varphi',
$$
and thus the Heat equation \eqref{eq-heat-moins} is transferred into  the  linear first order ODE  Cauchy problem
\begin{equation}
\label{cauchy-ordre1}
\varphi'=-\beta \frac{2r}{r^2+2k}\varphi,\quad \varphi(0)=1,
\end{equation}
which is solved as 
$$
\varphi(r)=\left(\frac{2k}{r^2+2k}\right)^{\beta},
$$
which in turn does satisfy \eqref{hyp-self} if $\beta\leq -1$ or $\beta\geq 0$. In order to keep nonconstant and bounded solutions, we now restrict to $\beta>0$: going back to \eqref{ansatz}, we are equipped, for any $\beta>0$, $\mu >0$, with solutions
\begin{equation}\label{self-similar}
u(t,x)=\mu \left(\frac{1}{\vert x\vert ^{2}+2kt}\right)^{\beta},
\end{equation}
and also, for any $\beta>0$, $\mu>0$, $\ep>0$,
\begin{equation}\label{self-similar-2}
u_\ep(t,x)=\mu \left(\frac{1}{\vert x\vert ^{2}+2kt+\ep}\right)^{\beta}.
\end{equation}

\begin{rem} Let $1\leq p\leq +\infty$. For any $t>0$, $u(t,\cdot)$ belongs to $L^p(\R^N)$ as soon as $\beta>\frac{N}{2p}$ and we have
$$
\Vert u(t,\cdot)\Vert_{L^p(\R^N)}=\frac{\mu C}{t^{\beta-\frac{N}{2p}}},
$$
with $C=C(\beta,p,N,k)>0$. In particular the $L^\infty$ norm decreases like $\frac 1{t^\beta}$.
\end{rem}

\subsection{Operator $\P ^{+}_k$}\label{ss:self-similar-plus}

Under assumption \eqref{hyp-self}, we have
$$
\P ^{+}_ku=\sum _{i=N-k+1}^{N}\lambda_i(D^{2}u)=\frac{1}{t^{\beta+1}}\varphi''+\frac{k-1}{t^{\beta+\frac 12}\vert x\vert}\varphi',
$$
and thus the Heat equation \eqref{eq-heat-plus} is transferred into  the linear problem
\begin{equation}
\label{cauchy-ordre2}
\varphi''+\frac{r^{2}+2(k-1)}{2r}\varphi'+\beta \varphi=0, \quad \varphi(0)=1,\quad  \varphi'(0)=0.
\end{equation}
One recognizes the ODE arising when looking after self-similar solutions to the Heat equation in dimension $k<N$. Hence, $\varphi(r):=e^{-\frac{r^{2}}{4}}$ solves the above problem provided that $\beta=\frac k 2$, and does satisfy \eqref{hyp-self}. Hence, going back to \eqref{ansatz}, we are equipped for any $\mu>0$, with solutions
\begin{equation}
\label{self-similar-plus}
u(t,x)=\frac{\mu}{t^{\frac k 2}}e^{-\frac{\vert x\vert ^{2}}{4t}}, \quad t>0,\, x\in \R^N.
\end{equation}
In particular notice that, for $\mu=(4\pi)^{-\frac N2}$, we have $\int _{\R^{N}}u(t,x)dx=t^{\frac{N-k}{2}}\to +\infty$, as $t\to +\infty$.

\begin{rem} For the self-containedness of the argument, we briefly discuss the problem \eqref{cauchy-ordre2} when $\beta\neq \frac k 2$. Using a Sturm-Liouville approach, one can recast the ODE problem \eqref{cauchy-ordre2} into an integral equation and prove the existence and uniqueness of a  local solution which moreover always satisfies $\varphi''(0)=-\frac \beta	k$, and is global when $\beta>0$, see \cite[Proposition 3.1]{Har-Wei-82}.

If $\beta=0$, the solution is $\varphi\equiv 1$.

If $\beta<0$, we claim that $\varphi\geq 1$: if not, from $\varphi(0)=1$, $\varphi'(0)=0$, $\varphi''(0)>0$, there must be a point $r_0>0$ where $\varphi$ reaches a local maximum larger than 1; testing the equation at $r_0$ yields a contradiction. We get rid of these solutions which are larger than one.

Now, for $\beta >0$, $\beta\neq \frac k 2$,  writing $\varphi(r)=e^{-\frac{r^2}{4}}\psi(\frac{r^2}{4})$, we see that $\psi$ has to solve 
 $$
 z\psi''(z)+\left(\frac k2-z\right)\psi'(z)+\left(\beta-\frac{k}{2}\right)\psi(z)=0, \quad \psi(0)=1,\quad \psi'(0)=1-\frac{2\beta}{k},
 $$
 and thus $\psi(z)={}_1F_1\left(\frac{k-2\beta}{2},\frac k 2,z\right)$, where ${}_1 F_1(a,b,z)$ is 
the confluent hypergeometric function of first kind, or Kummer's function, see \cite{Abr-Ste-64}. It is known that, when $a$ is not a nonpositive integer, 
$$
{}_1F_1(a,b,z)\sim \frac{\Gamma(b)}{\Gamma(a)}
\frac{e^z}{z^{b-a}}, \quad\text{ as } z\to +\infty.
$$
This transfers, when $0<\beta<\frac k 2$, into $\varphi(r)\sim \frac{C}{r^{2\beta}}$, for some $C>0$, as $r\to+\infty$,  and thus $\int _0 ^{+\infty}r^{N-1}\varphi(r)dr=+\infty$, so that these solutions are not \lq\lq admissible''. When $\beta>\frac  k2$ and $\frac{k-2\beta}{2}$ is not a negative integer, the conclusion is again $\varphi(r)\sim \frac{C}{r^{2\beta}}$, for some $C>0$ or $C<0$, as $r\to+\infty$, and these solutions are not \lq\lq admissible''. Last, when $\frac{k-2\beta}{2}$ is a negative integer, say $-p$, $\psi(z)$ is the $p$-th generalized Laguerre polynomial, which is known \cite[Section 6.31]{Sze-75} to change sign on $(0,+\infty)$, and thus these solutions are not \lq\lq admissible''.
\end{rem}

\subsection{On assumption \eqref{hyp-self}}\label{ss:on-assumption}

The goal of this short subsection is to show that assumption \eqref{hyp-self} is the one to be retained, as claimed above.

First, assuming the reverse inequality, namely
\begin{equation}
\label{hyp-reverse}
\varphi''(r)\leq \frac 1 r \varphi'(r), \quad \text{ for all } r>0,
\end{equation}
we can still compute $\mathcal P^{\pm} _k u$, where $u(t,x)$ is given by the self-similar ansatz \eqref{ansatz}. But, when dealing with operator $\P _k^{-}$, we now reach the second order ODE problem \eqref{cauchy-ordre2}, whose solution $\varphi(r)=e^{-\frac{r^{2}}{4}}$ does not satisfy \eqref{hyp-reverse}. Similarly, when dealing with operator $\P _k^{+}$, we now reach the first order ODE problem \eqref{cauchy-ordre1}, whose solutions $\varphi(r)=\left(\frac{2k}{r^{2}+2k}\right)^{\beta}$ do not satisfy \eqref{hyp-reverse}. 

Next, we may only assume the existence of $\ep>0$ such that
\begin{equation}
\label{hyp-court}
\varphi''(r)\geq \frac 1 r \varphi'(r), \quad \text{ for all } 0<r<\ep.
\end{equation}
Then, dealing with $\P _k ^{-}$, we reach $\varphi(r)=\left(\frac{2k}{r^{2}+2k}\right)^{\beta}$, say for $\beta>0$, for which $\varphi''(r)>\frac 1 r \varphi'(r)$ holds all along $(0,+\infty)$. In other words, we are back to assumption \eqref{hyp-self}. The same argument applies when dealing with $\P _k ^{+}$.
 
 Last, assuming \eqref{hyp-reverse} only a small bounded interval $(0,\ep)$, we reach a contradiction as in the case  of assumption \eqref{hyp-reverse}.

\section{Well-posedness of the different Cauchy problems}\label{s:well}

For $A\in \mathcal S_N(\R)$, we define $F_k^-(A):=\sum_{i=1}^k \lambda _i(A)$ and $F_k^+(A):=\sum_{i=N-k+1}^N\lambda _i(A)$. From the min-max theorem for eigenvalues of real symmetric matrices, we have that, for any $A,B\in \mathcal S _N(\R)$,
$$
A\geq B \Longrightarrow \lambda_i(A)\geq \lambda _i(B), \forall 1\leq i\leq N \Longrightarrow F_k^\pm(A)\geq F_k^\pm(B),
$$
and that, for any $A\in \mathcal S_N(\R)$, any $c\in \R$, 
$$
F_k^\pm(A+cId_N)-F_k^\pm(A)\geq kc.
$$

This enables to quote \cite[Theorem 2.7]{Cra-Lio-90}: for a initial data  $u_0\in UC(\R^N)$, the Cauchy problems associated with the Heat equations \eqref{eq-heat-moins} and \eqref{eq-heat-plus} admit a comparison principle and are globally well-posed, solutions being understood in the viscosity sense,  \cite{Cra-Lio-90}, \cite{Cra-Ish-Lio-92}, \cite{Eva-Spr-91}, \cite{Men-Qua-11}. The proof follows the three main steps: first prove a comparison principle using a  dedoubling variable method, next construct polynomial sub and supersolutions in the spirit of subsection \ref{ss:polynomial}, last conclude by the Perron's method.

By a straightforward and classical modification of the above procedure, one can prove the well-posedness of the Cauchy problems associated with equations \eqref{eq-fujita-moins} and \eqref{eq-fujita-plus}, at least {\it locally in time}. The main issue is then to determine if the local solution is global or blows up in finite time, which will be discussed in Section \ref{s:fujita}.

\section{The Heat equations: radial solutions of the Cauchy problems}\label{s:radial}

We  consider the Cauchy problem \eqref{eq-heat-moins} or \eqref{eq-heat-plus} starting from a radial initial data $u_0(x)=g(\vert x\vert)$, where $g:[0,+\infty)\to\R$. We suspect that $u(t,\cdot)$ remains radial for $t>0$ and therefore use the ansatz
$$
u(t,x)=\psi(t,\vert x\vert),
$$
for some $\psi=\psi(t,r)$. We compute the Hessian matrix and get
$$
D^{2}u=\frac{1}{\vert x\vert}(\partial _r \psi) Id_N-\left(\frac{1}{\vert x\vert}\partial _r \psi-\partial_{rr}\psi\right)\frac{x}{\vert x\vert}\otimes \frac{x}{\vert x\vert}.
$$
whose eigenvalues are $\frac{1}{\vert x\vert}\partial _r \psi$ with multiplicity (at least) $N-1$ and $\partial_{rr}\psi$ (see Section \ref{s:self-similar}). From now on, guided by \eqref{hyp-self} and subsection \ref{ss:on-assumption}, we assume
\begin{equation}
\label{hyp-cauchy}
\partial _{rr}\psi(t,r)\geq \frac 1 r \partial _r \psi(t,r), \text{ for all } t>0, r>0,
\end{equation}
which enables to compute $\mathcal P^{\pm} _k u$.

\subsection{Operator $\P ^{-}_k$}\label{ss:radial-moins}

Under assumption \eqref{hyp-cauchy}, we have
$$
\P ^{-}_ku=\sum _{i=1}^{k}\lambda_i(D^{2}u)=\frac{k}{\vert x\vert}\partial _r \psi,
$$
and thus the Heat equation \eqref{eq-heat-moins} is transferred into  the  linear transport equation
$$
\partial _t\psi=\frac k r \partial _r \psi,
$$
that can be solved via the method of characteristics. Indeed, for $r_0>0$, we have
$$
\frac{d}{dt}\left[\psi\left(t,(r_0^2-2kt)^{\frac 12}\right)\right]=0,
$$
and thus $\psi\left(t,(r_0^2-2kt)^{\frac 12}\right)=\psi(0,r_0)=g(r_0)$ which is recast
\begin{equation}\label{sol-transport-bis}
\psi(t,r)=g\left((2kt+r^{2})^{\frac{1}{2}}\right).
\end{equation}

Conversely, we need to check that assumption \eqref{hyp-cauchy} is satisfied. From \eqref{sol-transport-bis} we compute, assuming further regularity for $g$,
$$
\partial _{rr}\psi(t,r)-\frac 1 r \partial _r \psi(t,r)=\frac{r^2}{2kt+r^2}\left(g''(s)-\frac{1}{s}g'(s)\right),
$$
which we want to be nonnegative, and where we have let $s=(2kt+r^{2})^{\frac{1}{2}}$. 

As a conclusion, we have proved the following.


\begin{theo}[Radial solutions of the Cauchy problem \eqref{eq-heat-moins}] If $g:[0,+\infty)\to\R$  is twice differentiable on $(0,+\infty)$ and such that
\begin{equation}
\label{hyp-CI}
g''(s)-\frac  1 s g'(s)\geq 0, \forall s>0,
\end{equation}
then the solution of the Cauchy problem \eqref{eq-heat-moins} starting from $u_0(x)=g(\vert x\vert)$ is 
\begin{equation}
\label{sol-transport}
u(t,x)=g\left(\sqrt{2kt+\vert x\vert ^{2}}\right).
\end{equation}
\end{theo}

In other words, in the above situation, the so-called Heat equation \eqref{eq-heat-moins} does not diffuse but transports.  Let us investigate a few examples, for which we always assume $\mu >0$ and $\beta >0$.

\begin{example}\label{ex:exp} Function $g(s):=\mu e^{-\frac{ s ^{2}}{2k}}$ satisfies \eqref{hyp-CI}. From \eqref{sol-transport} we get the solution 
\begin{equation}
\label{separate-var}
u(t,x)=\mu e^{-t}e^{-\frac{\vert x\vert ^{2}}{2k}}.
\end{equation}
Notice that $u(t,x)=e^{-t}v(x)$ where $
v(x):=\mu e^{-\frac{\vert x\vert ^{2}}{2k}}$
is an eigenelement for operator $\P_k^{-}$: as noticed in \cite{Bir-Gal-Leo-17}, $v$ solves $\P _k^{-} v +v=0$.
\end{example}

\begin{example} Function $g(s):=\frac{\mu}{(\ep+s^{2})^\beta}$, $\ep>0$, satisfies \eqref{hyp-CI}. From \eqref{sol-transport} we recover the solution \eqref{self-similar-2}.
\end{example}

\begin{example}
Any function $g:[0,+\infty)\to \R$, twice differentiable on $(0,+\infty)$, which is nonincreasing and convex satisfies \eqref{hyp-CI}. In this framework $g(s)=\mu e^{-s}$ provides the solution $u(t,x)=\mu e^{-\sqrt{2kt+\vert x\vert ^{2}}}$ whereas $g(s)=\frac{\mu}{(\ep+s)^\beta}$, $\ep> 0$, provides the solution $u(t,x)=\frac{\mu}{\left(\ep+\sqrt{2kt+\vert x\vert ^{2}}\right)^\beta}$.
\end{example}

The appearance of a transport equation implies very striking phenomena for a so-called Heat equation: as shown by the following example, global extinction in finite time, or quenching, may occur.

\begin{example}[Quenching]\label{ex:quenching} Straightforward computations show that the smooth function
$$
g(s):=\begin{cases}e^{\frac{1}{s-1}} &\text{ if } 0\leq s<1\\
0 &\text{ if } s\geq 1\end{cases}
$$
does satisfy \eqref{hyp-CI}. Since $u_0(x)=g(\vert x\vert)$ is compactly supported in the ball of radius 1, the associated solution \eqref{sol-transport} of the Cauchy problem vanishes everywhere as soon as $t\geq \frac 1{2k}$, that is a quenching phenomena in finite time occurs.
\end{example}

\subsection{Operator $\P ^{+}_k$}\label{ss:plus}
Under assumption \eqref{hyp-cauchy}, we have
$$
\P ^{+}_ku=\sum _{i=N-k+1}^{N}\lambda_i(D^{2}u)=\partial_{rr}\psi+\frac{k-1}{\vert x\vert}\partial _r \psi
$$
and thus the Heat equation \eqref{eq-heat-plus} is transferred into  the linear convection diffusion equation
\begin{equation}
\label{to-be-solved}
\partial _t\psi=\partial_{rr}\psi+\frac {k-1} r \partial _r \psi.
\end{equation}
We assume that $g$ is bounded on $[0,+\infty)$. We denote by $\widetilde g$ its radial extension to $\R^{k}$, namely $ \widetilde g(x):=g(\vert x\vert)$ for $x\in \R^{k}$. We thus select
\begin{equation}
\label{select}
\psi(t,r)=\frac{1}{(4\pi t)^{\frac k 2}}\int _{\R^k} e^{-\frac{\vert r \dir -y\vert ^2}{4t}}  \widetilde g(y)dy=(G_k(t,\cdot)* \widetilde g)(r\dir ), \quad t>0, \, r\in \R,
\end{equation}
where $\dir$ is any unit vector in $\R ^k$. Since \eqref{to-be-solved} corresponds to solving the radial Heat equation in $\R^k$, the restriction of $\psi(t,r)$ to the $t>0$, $r>0$, solves \eqref{to-be-solved} and starts from $g(r)$.


As a conclusion, we have proved the following.

\begin{theo}[Radial solutions of the Cauchy problem \eqref{eq-heat-plus}]\label{th:radial-bis} If $g:[0,+\infty)\to \R$ is bounded and such that $\psi(t,r)$ given by \eqref{select} satisfies 
\begin{equation}
\label{hyp-cauchy-bis}
\partial _{rr}\psi(t,r)\geq \frac 1 r \partial _r \psi(t,r), \text{ for all } t>0, r>0,
\end{equation}
then the solution of the Cauchy problem \eqref{eq-heat-plus} starting from $u_0(x)=g(\vert x\vert)$ is
\begin{equation}
\label{sol-convolution}
u(t,x)=\psi(t,\vert x\vert)=\frac{1}{(4\pi t)^{\frac k 2}}\int _{\R ^k} e^{-\frac{ \vert(\vert x\vert \dir -y)\vert ^2}{4t}}  g(\vert y\vert)dy, \quad t>0,\, x\in \R^{N},
\end{equation}
where $\dir$ is any unit vector in $\R ^k$.
\end{theo}

Let us make a few comments. First, notice that $x$ lives in $\R ^N$ but we integrate over $y\in \R^k$. Next, observe that  \eqref{sol-convolution} does not provide a convolution formula for {\it any} radial solution, which would be in contrast with the fact that the equation is fully nonlinear. Actually, \eqref{sol-convolution} provides a convolution formula under condition \eqref{hyp-cauchy-bis} on the initial data, which is more consistent. Nonetheless, notice that \eqref{hyp-cauchy-bis} is stable by linear combination with nonnegative coefficients.

\begin{example}\label{ex:gauss}
For the Gaussian initial data $g(s)=\frac{1}{(4\pi a)^{\frac k 2}}e^{-\frac{s^2}{4a}}$, $a>0$, the convolution \eqref{select} is straightforwardly computed as $\psi(t,r)=\frac{1}{(4\pi(a+t))^{\frac k 2}}e^{-\frac{r^{2}}{4(a+t)}}$ which satisfies \eqref{hyp-cauchy-bis}. Hence we get the solution 
\begin{equation}\label{sol-gauss}
u(t,x)=\frac{1}{(4\pi(a+t))^{\frac k2}}e^{-\frac{\vert x\vert^{2}}{4(a+t)}}
\end{equation}
for $t>0$, $x\in \R^N$. Notice that, for any $1\leq p\leq +\infty$,
$$
\Vert u(t,\cdot)\Vert_{L^p(\R^N)}\sim \frac{C}{t^{\frac 12(k-\frac N p)}}, \quad \text{ as } t\to +\infty.
$$
where $C=C(a,p,N,k)>0$. In particular the $L^\infty$ norm decreases like $\frac{1}{t^{\frac k2}}$.
\end{example}

\begin{example} For the step function $g(s):=\mathbf 1_{(0,1)}(s)$ and $k=1$ (for simplicity), we use  \eqref{select} to compute $(\partial _{rr}\psi -\frac 1 r\partial _r \psi)(t,r)$ for $t>0$, $r>0$, and observe that it has the sign of 
$$
B(t,r):=\int _{-1}^{1}  (r(r-y)^{2}-2ty)
e^{-\frac{(r-y)^2}{4t}}dy.
$$
Using a formal calculation software, we get
$$
B(t,r)=2te^{-\frac{1+r^2}{4t}}\left(e^{\frac{r}{2t}}(-r+r^2+2t)-e^{-\frac{r}{2t}}(r+r^2+2t)\right),
$$
which fails to be nonnegative as soon as $0<r<1$, $0<t<\frac{r-r^2}2$. Hence \eqref{hyp-cauchy-bis} is not satisfied. Notice however that $g\not \in UC(\R^N)$ so that the well-posedness of the Cauchy problem is not obvious.
\end{example}

Now, we intend to provide examples of compactly supported initial data that satisfy \eqref{hyp-cauchy-bis}. This is more complicated than checking condition \eqref{hyp-CI}, which is the counterpart of condition \eqref{hyp-cauchy-bis} for  the Heat equation involving $\P _k ^-$, since the solution $\psi(t,r)$  is now given by a convolution, namely \eqref{select}. With additional assumptions on the initial data, we now try to find a \lq\lq local'' sufficient condition for \eqref{hyp-CI} to hold. Assuming that there is $\ep>0$ such that
\begin{equation}\label{tentative1}
\text{$g$ is of the class $C^{2}$ on $[0,\ep]$, nonnegative, $g'(0)=0$, 
and $g\equiv 0$ on $[\ep,+\infty)$,}
\end{equation}
the following computations are licit.  Formula \eqref{select} yields
$$
(4\pi t)^{\frac k2} \psi(t,r)=\int_{\R ^{k}}e^{-\frac{\vert re_1-y\vert ^2}{4t}}g(\vert y\vert)dy=\int_{\vert y\vert <\ep}e^{-\frac{\vert re_1-y\vert ^2}{4t}}g(\vert y\vert)dy,
$$
where $e_1$ denotes the first vector of the canonical basis of $\R^k$. In the sequel a generic $y\in \R^{k}$ is recast $y=(y_1,y')$ with $y_1\in \R$, $y'\in \R^{k-1}$. We differentiate with respect to $r$ and get, using  the shortcut $z=z(y):=re_1-y$,
\begin{eqnarray*}
(4\pi t)^{\frac k2} \psi _r (t,r)&=&\int_{\vert y\vert <\ep}\frac{-2(r-y_1)}{4t}
e^{-\frac{\vert z \vert ^2}{4t}}g(\vert y\vert)dy=-\int _{\vert y\vert <\ep} \frac{\partial}{\partial y_1}[e^{-\frac{\vert z\vert ^{2}}{4t}}]g(\vert y\vert)dy\\
&=&\int_{\vert y\vert <\ep}e^{-\frac{\vert z\vert ^{2}}{4t}}\frac{y_1}{\vert y\vert}g'(\vert y\vert)dy,
\end{eqnarray*}
using integration by part over $y_1$, noticing that the boundary terms vanishes since $g(\ep)=0$. Again we differentiate with respect to $r$, write $\int_{\vert y\vert <\ep}fdy=\int_{\vert y'\vert <\ep}\int_{-\alpha}^{\alpha}fdy_1dy'$ with the shortcut $\alpha=\alpha(y'):=\sqrt{\ep ^2-\vert y'\vert ^2}$, use integration by part over $y_1$  and reach
\begin{multline*}
(4\pi t)^{\frac k2} \psi _{rr} (t,r)=\int_{\vert y'\vert <\ep} e^{-\frac{\vert y'\vert ^2}{4t}}\frac{-\alpha g'(\ep)}{\ep}\left(e^{-\frac{(r-\alpha)^2}{4t}}+e^{-\frac{(r+\alpha)^2}{4t}}\right)dy'
\\
+\int_{\vert y\vert <\ep}e^{-\frac{\vert z\vert ^2}{4t}}\left(
\frac 1{\vert y\vert}g'(\vert y\vert)-\frac{y_1^{2}}{\vert y\vert^{3}}g'(\vert y\vert)+\frac{y_1^{2}}{\vert y \vert^{2}}g''(\vert y\vert)
\right)dy.
\end{multline*}
 Notice that the first integral term, over $\vert y'\vert <\ep$, is the boundary term. Putting all together we see that the sign of $\psi_{rr}(t,r)-\frac 1 r\psi_r(t,r)$ is that of
\begin{eqnarray}
I(t,r):=\int_{\vert y\vert <\ep}e^{-\frac{\vert z\vert ^2}{4t}}
\frac{y_1^{2}}{\vert y \vert^{2}}\left(g''(\vert y\vert)-\frac{1}{\vert y\vert}g'(\vert y\vert)\right)dy+\int_{\vert y\vert <\ep}e^{-\frac{\vert z\vert ^2}{4t}}\frac{r-y_1}{r\vert y\vert}g'(\vert y\vert)
dy\nonumber \\
+\int_{\vert y'\vert <\ep} e^{-\frac{\vert y'\vert ^2}{4t}}\frac{-\alpha g'(\ep)}{\ep}\left(e^{-\frac{(r-\alpha)^2}{4t}}+e^{-\frac{(r+\alpha)^2}{4t}}\right)dy'=:(I_1+I_2+I_3)(t,r).\nonumber
\end{eqnarray}
Using again integration by part over $y_1$ we get
\begin{eqnarray*}
I_2(t,r)=\int_{\vert y'\vert <\ep}e^{-\frac{\vert y'\vert ^2}{4t}}\frac{2t}{r}
\left(e^{-\frac{(r-\alpha)^2}{4t}}\frac{g'(\ep)}{\ep}-e^{-\frac{(r+\alpha)^2}{4t}}\frac{g'(\ep)}{\ep}\right)dy'\\
-\frac{2t}{r}\int_{\vert y\vert <\ep}e^{-\frac{\vert z\vert^2}{4t}}\frac{y_1}{\vert y\vert ^2}\left(g''(\vert y\vert)-\frac{1}{\vert y\vert}g'(\vert y\vert)\right)dy.
\end{eqnarray*}
Putting all together we arrive at
\begin{eqnarray*}
I(t,r)&=& \int_{\vert y\vert <\ep}e^{-\frac{\vert z\vert^2}{4t}}\frac{y_1(y_1-\frac{2t}{r})}{\vert y\vert ^2}\left(g''(\vert y\vert)-\frac{1}{\vert y\vert}g'(\vert y\vert)\right)dy\\
&&+\int _{\vert y'\vert <\ep}
e^{-\frac{\vert y'\vert ^2}{4t}}
\frac{t}{r}\frac{-g'(\ep)}{\ep}e^{-\frac{(r+\alpha)^{2}}{4t}}\left(e^{\alpha r t^{-1}}(\alpha r t^{-1}-2)+\alpha r t^{-1}+2
\right)dy'\\
&=:&(J_1+J_2)(t,r).
\end{eqnarray*}
We easily see that $e^\lambda (\lambda -2)+\lambda +2\geq 0$ for all $\lambda \geq 0$ and, since $g'(\ep)\leq 0$, we have $J_2(t,r)\geq 0$ for all $t>0$, $r>0$. Nonetheless even if we assume
\begin{equation}
\label{tentative2}
g''(s)-\frac 1 sg'(s)\geq 0, \text{ for all }0< s< \ep,
\end{equation}
we cannot hope the term $J_1(t,r)$ to remain nonegative for all $t>0$, $r>0$ --- unless it vanishes--- because of the term $y_1(y_1-\frac{2t}{r})$. This is a strong indication that the nonnegative initial data for which \eqref{hyp-cauchy-bis} holds are rather \lq\lq rare'' or, in other words and roughly speaking, condition \eqref{hyp-cauchy-bis} seems to be very \lq\lq unstable''. In particular for $k=1$ or $g'(\ep)=0$, the nonnegative favorable term $J_2(t,r)$ vanishes.

Nevertheless,  assuming equality in \eqref{tentative2} obviously saves the day and provides the following example, which is an important tool for the proof of Theorem \ref{th:fujita-plus} on the Fujta blow-up phenomena.
 
\begin{example}\label{ex:compactly} Let $\ep>0$ be given. Function $g(s):=(\ep ^2-s^{2})_+$ clearly satisfies \eqref{tentative1} and the equality in \eqref{tentative2}, so that the solution of the Cauchy problem \eqref{eq-heat-plus} starting from $u_0(x)=g(\vert x\vert)$ is given by the convolution formula \eqref{sol-convolution}.
\end{example}

\begin{rem}
From Example \ref{ex:gauss}, Example \ref{ex:compactly} and the comparison principle we deduce that, for any nonnegative and nontrivial initial data $u_0$ (not necessarily radial) having tails that can be dominated by a Gaussian tail, the solution $u(t,x)$ of the Heat equation \eqref{eq-heat-plus} starting from $u_0$ satisfies
$$
\frac{C_1}{(1+t)^{\frac k 2}}\leq \Vert u(t,\cdot)\Vert _{L^\infty(\R^N)}\leq \frac{C_2}{(1+t)^{\frac k 2}}, \text{ for all } t\geq 0,
$$
for some positive constants $C_1=C_1(u_0)$, $C_2=C_2(u_0)$.
\end{rem}

\begin{rem}
Assume that the initial data $u_0(x)=g(\vert x\vert)$ is such that the conclusion of Theorem \ref{th:radial-bis} holds, and that $g_i:=\int_{\R^k}\vert x\vert ^{i} g(\vert x\vert)dx<+\infty$ for $i=0,1$. Then the solution  becomes asymptotically self-similar in the sense that,
\begin{equation*}\label{bidule}
\left \| u(t,\cdot)-g_0\frac{1}{(4\pi t)^{\frac k2}}e^{-\frac{\vert \, \cdot\,  \vert ^2}{4t}}\right \|_{L^\infty(\R ^N)}=o\left(t^{-\frac{k}{2}}\right)\;\text{ as $t\to+\infty$}.
\end{equation*} 
This can be proved from the convolution formula \eqref{sol-convolution} by reproducing the standard argument for the (classical) Heat equation, see  for instance  the monograph of Giga, Giga and Saal \cite[subsection 1.1.5]{Gig-Gig-Saa-10}.
\end{rem}

\section{Global vs blow-up solutions for the doubly nonlinear Cauchy problems}\label{s:fujita}

In this section, as explained in Section \ref{s:well}, we wonder if the local solution to the Cauchy problem associated with equations \eqref{eq-fujita-moins} or \eqref{eq-fujita-plus} is global or not. 

Let us recall that, in his seminal work \cite{Fuj-66}, Fujita considered solutions $u(t,x)$ to the nonlinear ($p>0$) Heat equation 
\begin{equation}
\label{eq-fujita-66}
\partial _t u=\Delta u+u^{1+p}\quad \text{ in } (0,\infty)\times \R^{N},
\end{equation}
supplemented with a nonnegative and nontrivial initial data and proved the following: when $0<p<\frac 2 N$, any  solution blows up in finite time whereas, when $p>\frac 2N$ some solutions with small initial data are global in time. Hence, for equation \eqref{eq-fujita-66}, $p_F:=\frac 2 N$ is the so-called Fujita exponent. Let us observe that, as well-known, solutions to the Heat equation  $\partial _t u=\Delta u$ tend to zero as $t\to\infty$ like  $\mathcal O\left(t^{-\frac  N 2}\right)$, which is a formal  argument to guess $p_F=\frac 2 N$.

In the sequel we prove that $p_F=0$ for equation \eqref{eq-fujita-moins} involving $\P _k ^{-}$, whereas $p_F=\frac 2 k$ for equation \eqref{eq-fujita-plus} involving $\P _k^{+}$.

\subsection{Operator $\P _k^-$}\label{ss:fujita-moins}

As seen in Example \ref{ex:exp}, the $L^{\infty}$ norm of some solutions to the Heat equation \eqref{eq-heat-moins} decrease exponentially fast to zero at large times. This is a strong indication that the Fujita exponent is $p_F=0$. 

\begin{prop}[Some global solutions with light tails] Let $p>0$ be given. Assume $0\leq u_0(x)\leq Ce^{-\frac{\vert x\vert ^{2}}{2k}}$ for some $0< C\leq 1$. Then the solution to \eqref{eq-fujita-moins} starting from $u_0$ is global in time and satisfies
$$
\Vert u(t,\cdot)\Vert_{L^{\infty}(\R^{N})}\leq \begin{cases}\frac{C}{(1-C^p)^{1/p}}e^{-t} & \text{ if } 0<C<1\\
C & \text{ if } C=1. \end{cases}
$$
\end{prop}

\begin{proof}
 We define $v(t,x):=Ce^{-t}e^{-\frac{\vert x\vert ^{2}}{2k}}$ which is the solution of the Heat equation \eqref{eq-heat-moins} starting from $v_0(x)=Ce^{-\frac{\vert x\vert ^{2}}{2k}}$. We look after a supersolution to \eqref{eq-fujita-moins} in the form 
 $$
\overline  u(t,x)=f(t)v(t,x),
 $$
 with $f$ to be chosen and starting from $f(0)=1$. We compute
 $$
 (\partial _t \overline u -\P _k^{-}\overline u -\overline u^{1+p})(t,x)=f'(t)v(t,x)-f^{1+p}(t)v^{1+p}(t,x)
 $$
 which is nonnegative provided
 $$
 \frac{f'(t)}{f^{1+p}(t)}\geq \Vert v(t,\cdot)\Vert _{L^{\infty}}^{p}=C^{p}e^{-pt}.
 $$
 Since $C\leq 1$ the Cauchy problem $\frac{f'(t)}{f^{1+p}(t)}=C^pe^{-pt}$, $f(0)=1$ is globally solved as 
 $$
 f(t)=\frac{1}{(1+C^{p}(e^{-pt}-1))^{1/p}}\leq \begin{cases} \frac{1}{(1-C^{p})^{1/p}} & \text{ if } 0<C<1\\
 e^{t} & \text{ if } C=1.\end{cases}
 $$
 From the comparison principle, we deduce $0\leq u(t,x)\leq \overline u(t,x)$ for all $t>0$, $x\in \R^N$, which provides the result.
\end{proof}

The solutions \eqref{self-similar-2} to the Heat equation \eqref{eq-heat-moins} provide examples of global solutions to \eqref{eq-fujita-moins} with initial heavy tails, provided $p$ is large enough.

\begin{prop}[Some global solutions with heavy tails] Let $\beta>0$ be given. Let $p>\frac 1 \beta$ be given. Assume $0\leq u_0(x)\leq \frac{C}{(\vert x\vert^{2}+\ep)^{\beta}}$ for some $C>0$, $\ep>0$ satisfying
$$
\frac{C^{p}}{\ep^{p\beta-1}}\leq 2k\frac{p\beta-1}{p}.
$$
Then the solution to \eqref{eq-fujita-moins} starting from $u_0$ is global in time and satisfies
$$
\Vert u(t,\cdot)\Vert_{L^{\infty}(\R^{N})}\leq \begin{cases}\frac{C'}{t^{\beta}} & \text{ if } \frac{C^{p}}{\ep^{p\beta-1}}< 2k\frac{p\beta-1}{p}\medskip \\
\frac{C'}{t^{1/p}} & \text{ if } \frac{C^{p}}{\ep^{p\beta-1}}= 2k\frac{p\beta-1}{p},\end{cases}
$$
for some $C'=C'(\beta,p,k,\ep,C)>0$.
\end{prop}

\begin{proof} We define $v(t,x):=\frac{C}{(\vert x\vert ^2+2kt+\ep)^\beta}$ which is the solution of the Heat equation starting from $v_0(x)=\frac{C}{(\vert x\vert ^2+\ep)^\beta}$. Next, the proof is similar as the previous one.
\end{proof}

From any of the two above propositions, we thus conclude that we do have $p_F=0$. Notice also that $p_F=0$ also follows from the following observation  from \cite{Bir-Gal-Leo-17}: for any $p>0$, equation \eqref{eq-fujita-moins} admits the stationary solutions
$$
\left(\frac{2k}{p(\mu+\vert x\vert ^{2})}\right)^{\frac{1}{p}},
$$
which corresponds to the critical case $p=\frac 1\beta$ of the above proposition.

\subsection{Operator $\P _k^+$}\label{ss:fujita-plus}

As seen in Example \ref{ex:gauss}, the $L^{\infty}$ norm of some solutions to the Heat equation \eqref{eq-heat-plus} decrease like $t^{-\frac k 2}$ at large times. This is an indication that the Fujita exponent  $p_F$ is smaller than $\frac 2 k$. This is confirmed by the following construction of global solutions when $p>\frac 2 k$.

\begin{prop}[Some global solutions when $p>\frac 2 k$] Assume $p>\frac 2 k$. Let $a>0$ be given.  Assume $0\leq u_0(x)\leq \frac{C}{(4\pi a)^{\frac k 2}}e^{-\frac{\vert x\vert ^{2}}{4a}}$ for some $C>0$. Then, if $C>0$ is small enough, the solution to \eqref{eq-fujita-plus} starting from $u_0$ is global in time and satisfies
$$
\Vert u(t,\cdot)\Vert_{L^{\infty}(\R^{N})}\leq  \frac{C'}{(a+t)^{\frac k 2}}
$$
for some $C'=C'(p,k,a,C)>0$.
\end{prop}

\begin{proof}
 We define $v(t,x):=\frac{C}{(4\pi(a+t))^{\frac k2}}e^{-\frac{\vert x\vert^{2}}{4(a+t)}}$ which is the solution of the Heat equation \eqref{eq-heat-plus} starting from $v_0(x)=\frac{C}{(4\pi a)^{\frac k2}}e^{-\frac{\vert x\vert^{2}}{4a}}$. We look after a supersolution to \eqref{eq-fujita-plus} in the form 
 $$
\overline  u(t,x)=f(t)v(t,x),
 $$
 with $f$ to be chosen and starting from $f(0)=1$. We compute
 $$
 (\partial _t \overline u -\P _k^{+}\overline u -\overline u^{1+p})(t,x)=f'(t)v(t,x)-f^{1+p}(t)v^{1+p}(t,x)
 $$
 which is nonnegative provided
 $$
 \frac{f'(t)}{f^{1+p}(t)}\geq \Vert v(t,\cdot)\Vert _{L^{\infty}}^{p}=\frac{C^{p}}{(4\pi(a+t))^{\frac{pk}{2}}}.
 $$
If $C>0 $ is sufficiently small, the Cauchy problem $\frac{f'(t)}{f^{1+p}(t)}=\frac{C^{p}}{(4\pi(a+t))^{\frac{pk}{2}}}$, $f(0)=1$ is globally solved as 
 $$
 f(t)=\frac{1}{\left(1+\frac{pC^{p}}{(4\pi)^{\frac{pk}{2}}(\frac{pk}{2}-1)}\left(\frac{1}{(a+t)^{\frac{pk}{2}-1}}-\frac{1}{a^{\frac{pk}{2}-1}}\right)\right)^{\frac 1p}}\leq \frac{1}{\left(1-\frac{pC^{p}}{(4\pi)^{\frac{pk}{2}}(\frac{pk}{2}-1)a^{\frac{pk}{2}-1}}\right)^{\frac 1p}}.
 $$
 From the comparison principle, we deduce $0\leq u(t,x)\leq \overline u(t,x)$ for all $t>0$, $x\in \R^N$, which provides the result.
\end{proof}

Our last main result shows that $p_F=\frac 2k$.

\begin{theo}[Systematic blow-up when $p<\frac 2 k$]\label{th:fujita-plus} Assume $0<p<\frac 2 k$. Then for any $u_0\in UC(\R^{N})$ nonnegative and nontrivial, the solution to \eqref{eq-fujita-plus} starting from $u_0$ blows up in finite time.
\end{theo}

\begin{proof} Since the equation is invariant by translation in space and in view of the comparison principle, it is enough to consider the case of the compactly supported initial data
$$
u_0(x)=g(\vert x\vert):=(\ep ^2-\vert x\vert ^2)_+, \quad x\in \R^N,
$$
for a arbitrary small $\ep>0$. We assume that the solution $u(t,x)$ ($t>0$, $x\in \R^N$)  to \eqref{eq-fujita-plus} starting from $u_0$ is global in time and look after a contradiction. To start with, we make the additional assumption (to be removed in the end of the proof) that the viscosity solution is radial and smooth, in the sense that $u(t,x)=\varphi(t,\vert x\vert)$ for some $\varphi=\varphi(t,r)$ smooth on $(0,+\infty)\times[0,+\infty)$.

In some related proofs of blow-up phenomena, see \cite{Kap-63}, \cite{Fuj-66}, \cite{Alf-fuj}, the fundamental solution of the underlying linear Heat equation is used. We are not equipped with such a tool but it turns out that the solution of Example \ref{ex:compactly} has enough good properties for a modification of the argument to apply. Hence, we denote by $v(t,x)$ ($t>0$, $x\in \R ^N$) the solution to \eqref{eq-heat-plus} starting from $u_0$, as provided by Theorem \ref{th:radial-bis} and Example \ref{ex:compactly}. In particular we have $v(t,x)=\psi(t,\vert x\vert)$ for  $\psi=\psi(t,r)$ provided by the convolution formula \eqref{select} and smooth on $(0,+\infty)\times[0,+\infty)$.

We define the quantity (notice that we integrate over $z\in \R^k$)
$$
f(t):=\int _{\R ^k}v(t,\vert z\vert e_1)u_0(\vert z\vert e_1)dz=\int _{\R ^k}\psi(t,\vert z\vert)g(\vert z\vert)dz,
$$
where $e_1$ is the first unit vector of the canonical basis of $\R^N$. We aim at finding estimates of $f(t)$ from below and above which are incompatible as $t\to+\infty$. 

From the expression of the initial data, we have
$$
f(t)\geq \frac{\ep^2}{2}\int_{\vert z\vert <\ep/\sqrt 2}v(t,\vert z\vert e_1)dz.
$$
Since $v(t,x)$ is given by the convolution formula \eqref{sol-convolution} we see, from the expression of the initial data, that, for any $\vert x\vert <\ep/\sqrt 2$ and $t\geq 1$, $v(t,x)\geq \frac{C}{t^{\frac k 2}}$ for some $C=C(\ep)>0$. As a result, we reach the estimate from below
\begin{equation}
\label{below}
f(t)\geq \frac{C_1}{t^{\frac k 2}}, \quad \forall t\geq  1,
\end{equation}
for some $C_1=C_1(\ep)>0$.

Next, for a given $t>0$ and any small $\alpha>0$, we let
$$
g(s):=\int_{\R^{k}} v(t-s+\alpha,\vert z\vert e_1)u(s,\vert z\vert e_1)dz, \quad 0\leq s\leq t.
$$
We differentiate with respect to $s$ and use the equations satisfied by $v$ and $u$ to reach
\begin{eqnarray*}
g'(s)&=&\int_{\R ^{k}}\left(-\P _k^{+}v(t-s+\alpha,\vert z\vert e_1) u(s,\vert z\vert e_1)+v(t-s+\alpha,\vert z\vert e_1)\P _k ^{+}u(s,\vert z\vert e_1)\right)dz\\
 &&+\int_{\R ^{k}}v(t-s+\alpha,\vert z\vert e_1)u^{1+p}(s,\vert z\vert e_1)dz=:h_1(s)+h_2(s).
\end{eqnarray*}
A first key point is that, as understood in Section \ref{s:radial} and roughly speaking, $\P _k^{+}v$ corresponds to the Laplacian in dimension $k<N$. Another crucial point is that, for a radial fonction $u$, $\mathcal P_k^{+}u$ is always larger than the Laplacian in dimension $k<N$, this following from the beginning of Section \ref{s:radial}. Precisely, denoting $\vert S_{k-1}\vert$ the area of the unit hypersphere of $\R^{k}$, we have
\begin{eqnarray*}
h_1(s)&\geq& \vert S_{k-1}\vert \int_0^{+\infty}\Big((-\psi_{rr}-\frac{k-1}{r}\psi_r)(t-s+\alpha,r)\varphi(s,r)\\
&&+ \psi(s,r)(\varphi_{rr}+\frac{k-1}{r}\varphi_r)(t-s+\alpha,r)\Big)r^{k-1}dr
\end{eqnarray*}
which is nonnegative as seen by integrating by parts. Next, from the convolution formula \eqref{sol-convolution} and Fubini-Tonelli theorem, we see that, for all $\tau>0$,
$$
\int _{\R^{k}} v(\tau,\vert z\vert e_1)dz=\Vert g(\vert \cdot \vert)\Vert _{L^{1}(\R ^{k})}=:C(\ep,k)^{\frac 1 p}=C^{\frac 1 p}>0.
$$
Therefore we have, from Jensen inequality,
$$
g'(s)\geq h_2(s)\geq C\left( \int _{\R ^{k}} v(t-s+\alpha,\vert z\vert e_1)u(s,\vert z\vert e_1) dz\right)^{1+p}=Cg^{1+p}(s).
$$
Integrating this differential inequality from $0$ to $t$, we get $Ct\leq \frac{-1}{p}\left(\frac{1}{g^p(t)}-\frac{1}{g^p(0)}\right)\leq \frac{1}{pg^p(0)}$. Now letting $\alpha \to 0$, this is recast
\begin{equation}
\label{above}
f(t)\leq \frac{C_2}{t^{\frac 1p}}, \quad \forall t\geq 1,
\end{equation}
for some $C_2=C_2(p,k,\ep)>0$. As announced, letting $t\to +\infty$ into \eqref{below} and \eqref{above} contradicts $0<p<\frac 2 k$.

It remains to remove the assumption that $u$ is radial and smooth, which can be done thanks to the comparison principle and the crucial point mentioned above concerning radial solutions. Indeed, let us denote by $w(t,x)$ the solution to
$$
\partial _t w=\Delta w +w^{1+p} \quad \text{ in } (0,+\infty)\times \R ^{k},
$$
starting from $w_0(x)=u_0(x)=g(\vert x\vert)$, for which we know that $w(t,x)=\theta(t,\vert x\vert)$ for some $\theta=\theta(t,r)$ smooth on $(0,+\infty)\times [0,+\infty)$. We switch to $\R ^{N}$ by letting
$$
\underline w(t,x):=\theta(t,\vert x\vert e_1), \quad t>0,\, x\in \R ^N.
$$
Since $\Delta \underline w\leq \mathcal P _k ^+\underline w$, we deduce from the comparison principle that $\underline w \leq u$, and it suffices to prove the blow-up of $\underline w$. Since $\underline w$ possesses all the necessary properties, we can reproduce the above argument with $\underline w$ playing the role of $u$. This concludes the proof of Theorem \ref{th:fujita-plus}.
\end{proof}

\bibliographystyle{siam}  

\bibliography{biblio}

\def\cprime{$'$}
\begin{thebibliography}{10}

\bibitem{Abr-Ste-64}
{\sc M.~Abramowitz and I.~A. Stegun}, {\em Handbook of mathematical functions
  with formulas, graphs, and mathematical tables}, vol.~55 of National Bureau
  of Standards Applied Mathematics Series, For sale by the Superintendent of
  Documents, U.S. Government Printing Office, Washington, D.C., 1964.

\bibitem{Alf-fuj}
{\sc M.~Alfaro}, {\em Fujita blow up phenomena and hair trigger effect: the
  role of dispersal tails}, Ann. Inst. H. Poincar\'{e} Anal. Non Lin\'{e}aire,
  34 (2017), pp.~1309--1327.

\bibitem{Amb-Son-96}
{\sc L.~Ambrosio and H.~M. Soner}, {\em Level set approach to mean curvature
  flow in arbitrary codimension}, J. Differential Geom., 43 (1996),
  pp.~693--737.

\bibitem{Bir-Gal-Ish-18}
{\sc I.~Birindelli, G.~Galise, and H.~Ishii}, {\em A family of degenerate
  elliptic operators: maximum principle and its consequences}, Ann. Inst. H.
  Poincar\'{e} Anal. Non Lin\'{e}aire, 35 (2018), pp.~417--441.

\bibitem{Bir-Gal-Ish-preprint}
\leavevmode\vrule height 2pt depth -1.6pt width 23pt, {\em Towards a reversed
  {F}aber-{K}rahn inequality for the truncated laplacian}, arXiv preprint
  arXiv:1803.07362,  (2018).

\bibitem{Bir-Gal-Leo-17}
{\sc I.~Birindelli, G.~Galise, and F.~Leoni}, {\em Liouville theorems for a
  family of very degenerate elliptic nonlinear operators}, Nonlinear Anal., 161
  (2017), pp.~198--211.

\bibitem{BlaEsteRos}
{\sc P.~Blanc, C.~Esteve, and J.~D. Rossi}, {\em The evolution problem
  associated with eigenvalues of the {H}essian}, arXiv preprint
  arXiv:1901.01052,  (2019).

\bibitem{BlaRos}
{\sc P.~Blanc and J.~D. Rossi}, {\em Games for eigenvalues of the {H}essian and
  concave/convex envelopes}, arXiv preprint arXiv:1801.03383,  (2018).

\bibitem{CafLiNirIII}
{\sc L.~Caffarelli, Y.~Li, and L.~Nirenberg}, {\em Some remarks on singular
  solutions of nonlinear elliptic equations {III}: viscosity solutions
  including parabolic operators}, Comm. Pure Appl. Math., 66 (2013),
  pp.~109--143.

\bibitem{Cra-Ish-Lio-92}
{\sc M.~G. Crandall, H.~Ishii, and P.-L. Lions}, {\em User's guide to viscosity
  solutions of second order partial differential equations}, Bull. Amer. Math.
  Soc. (N.S.), 27 (1992), pp.~1--67.

\bibitem{Cra-Lio-90}
{\sc M.~G. Crandall and P.-L. Lions}, {\em Quadratic growth of solutions of
  fully nonlinear second order equations in {${\bf R}^n$}}, Differential
  Integral Equations, 3 (1990), pp.~601--616.

\bibitem{Eva-Spr-91}
{\sc L.~C. Evans and J.~Spruck}, {\em Motion of level sets by mean curvature.
  {I}}, J. Differential Geom., 33 (1991), pp.~635--681.

\bibitem{Fuj-66}
{\sc H.~Fujita}, {\em On the blowing up of solutions of the {C}auchy problem
  for {$u_{t}=\Delta u+u^{1+\alpha }$}}, J. Fac. Sci. Univ. Tokyo Sect. I, 13
  (1966), pp.~109--124.

\bibitem{Gal-19}
{\sc G.~Galise}, {\em On positive solutions of fully nonlinear degenerate
  {L}ane--{E}mden type equations}, J. Differential Equations, 266 (2019),
  pp.~1675--1697.

\bibitem{Gig-Gig-Saa-10}
{\sc M.-H. Giga, Y.~Giga, and J.~Saal}, {\em Nonlinear partial differential
  equations}, vol.~79 of Progress in Nonlinear Differential Equations and their
  Applications, Birkh\"{a}user Boston, Inc., Boston, MA, 2010.
\newblock Asymptotic behavior of solutions and self-similar solutions.

\bibitem{Gui-Ni-Wan-92}
{\sc C.~Gui, W.-M. Ni, and X.~Wang}, {\em On the stability and instability of
  positive steady states of a semilinear heat equation in {${\bf R}^n$}}, Comm.
  Pure Appl. Math., 45 (1992), pp.~1153--1181.

\bibitem{Har-Wei-82}
{\sc A.~Haraux and F.~B. Weissler}, {\em Nonuniqueness for a semilinear initial
  value problem}, Indiana Univ. Math. J., 31 (1982), pp.~167--189.

\bibitem{HarvLawCPAM}
{\sc F.~R. Harvey and H.~B. Lawson, Jr.}, {\em Dirichlet duality and the
  nonlinear {D}irichlet problem}, Comm. Pure Appl. Math., 62 (2009),
  pp.~396--443.

\bibitem{HarvLawIUMJ}
\leavevmode\vrule height 2pt depth -1.6pt width 23pt, {\em {$p$}-convexity,
  {$p$}-plurisubharmonicity and the {L}evi problem}, Indiana Univ. Math. J., 62
  (2013), pp.~149--169.

\bibitem{Kap-63}
{\sc S.~Kaplan}, {\em On the growth of solutions of quasi-linear parabolic
  equations}, Comm. Pure Appl. Math., 16 (1963), pp.~305--330.

\bibitem{Kav-87}
{\sc O.~Kavian}, {\em Remarks on the large time behaviour of a nonlinear
  diffusion equation}, Ann. Inst. H. Poincar\'{e} Anal. Non Lin\'{e}aire, 4
  (1987), pp.~423--452.

\bibitem{Men-Qua-11}
{\sc R.~Meneses and A.~Quaas}, {\em Fujita type exponent for fully nonlinear
  parabolic equations and existence results}, J. Math. Anal. Appl., 376 (2011),
  pp.~514--527.

\bibitem{OberSilv}
{\sc A.~M. Oberman and L.~Silvestre}, {\em The {D}irichlet problem for the
  convex envelope}, Trans. Amer. Math. Soc., 363 (2011), pp.~5871--5886.

\bibitem{ShaJP}
{\sc J.-P. Sha}, {\em {$p$}-convex {R}iemannian manifolds}, Invent. Math., 83
  (1986), pp.~437--447.

\bibitem{Sze-75}
{\sc G.~Szeg\H{o}}, {\em Orthogonal polynomials}, American Mathematical
  Society, Providence, R.I., fourth~ed., 1975.
\newblock American Mathematical Society, Colloquium Publications, Vol. XXIII.

\bibitem{Wei-81}
{\sc F.~B. Weissler}, {\em Existence and nonexistence of global solutions for a
  semilinear heat equation}, Israel J. Math., 38 (1981), pp.~29--40.

\bibitem{WuH}
{\sc H.~Wu}, {\em Manifolds of partially positive curvature}, Indiana Univ.
  Math. J., 36 (1987), pp.~525--548.

\end{thebibliography}

\end{document}